\documentclass[a4paper,10pt]{article}

\usepackage{amsmath,amsfonts,amssymb,amsthm}
\usepackage{graphicx}
\usepackage{enumerate}
\usepackage{xcolor}
\usepackage{url}
\usepackage{tcolorbox}

\usepackage{hyperref}

\DeclareMathOperator{\im}{im} 

\DeclareMathOperator{\diag}{diag}
\DeclareMathOperator{\e}{e} 
\DeclareMathOperator{\id}{Id}

\newcommand{\R}{{\mathbb R}}
\newcommand{\N}{{\mathbb N}}

\newcommand{\dd}[2]{\frac{\text{d} #1}{\text{d} #2}}
\newcommand{\trans}{\mathsf{T}}
\newcommand{\ones}{1}

\DeclareMathOperator{\relint}{relint}

\newcommand{\dep}{d}
\newcommand{\rka}{l}

\newcommand{\IL}{I}
\newcommand{\JL}{J}
\newcommand{\LL}{L}

\newcommand{\mBp}{{\mathcal B}}
\newcommand{\mG}{G}
\newcommand{\mH}{H}

\newcommand{\ce}[1]{{\sf{#1}}}

\newtheorem{thm}{Theorem}
\newtheorem{pro}[thm]{Proposition}
\newtheorem{lem}[thm]{Lemma}
\newtheorem{cor}[thm]{Corollary}
\newtheorem{fac}[thm]{Fact}

\theoremstyle{definition}

\newtheorem{rem}[thm]{Remark}
\newtheorem{exa}{Example}

\newcommand{\gray}[1]{{\color{gray}#1}}

\newcommand\blfootnote[1]{%
  \begin{NoHyper}
  \renewcommand\thefootnote{}\footnote{#1}%
  \addtocounter{footnote}{-1}%
  \end{NoHyper}
}

\begin{document}

\title{
Parametrized systems of \\ generalized polynomial inequalities \\ 
via linear algebra and convex geometry
}

\author{%
Stefan M\"uller$$,
Georg Regensburger
}

\date{\today}

\maketitle

\begin{abstract}
We provide fundamental results on positive solutions to 
parametrized systems of generalized polynomial {\em inequalities} (with real exponents and positive parameters),
including generalized polynomial {\em equations}.
In doing so, we also offer a new perspective on fewnomials and (generalized) mass-action systems.

We find that geometric objects, rather than matrices, determine generalized polynomial systems:
a bounded set/``polytope'' $P$ (arising from the coefficient matrix)
and two subspaces representing monomial differences and dependencies (arising from the exponent matrix).
The dimension of the latter subspace, the monomial dependency $d$, is crucial.

As our main result, we rewrite {\em polynomial inequalities}
in terms of $d$ {\em binomial equations} on $P$, involving $d$ monomials in the parameters.
In particular, we establish an explicit bijection between the original solution set 
and the solution set on $P$ via exponentiation.

(i) Our results apply to any generalized polynomial system.
(ii) The dependency $d$ and the dimension of $P$ indicate the complexity of a system.
(iii) Our results are based on methods from linear algebra and con\-vex/poly\-hedral geometry,
and the solution set on $P$ can be further studied using methods from analysis
such as sign-characteristic functions (introduced in this work).

We illustrate our results (in particular, the relevant geometric objects) 
through three examples from real fewnomial and reaction network theory.
For two mass-action systems,
we parametrize the set of equilibria and the region for multistationarity, respectively,
and even for univariate trinomials, we offer new insights:
We provide a ``solution formula'' involving discriminants and ``roots''.

%

\vspace{2ex}
\end{abstract}

\vspace{-2ex}
\blfootnote{
\scriptsize

\noindent
{\bf Stefan M\"uller} \\
Faculty of Mathematics, University of Vienna, Oskar-Morgenstern-Platz 1, 1090 Wien, Austria \\[1ex]
{\bf Georg Regensburger} \\
Institut f\"ur Mathematik, Universit\"at Kassel, Heinrich-Plett-Strasse 40, 34132 Kassel, Germany \\[1ex]
Corresponding author: 
\href{mailto:st.@univie.ac.at}{st.mueller@univie.ac.at}
}


\section{Introduction}

In this conceptual paper,
we provide fundamental results on positive solutions to 
pa\-ram\-e\-trized systems of generalized polynomial {\em inequalities} (with real exponents and positive parameters),
including generalized polynomial {\em equations}.
In doing so, we also offer a new perspective on fewnomials and (generalized) mass-action systems.

Let 
$A' \in \R^{\rka' \times m}$ and
$A'' \in \R^{\rka'' \times m}$
be coefficient matrices,
$B \in \R^{n \times m}$ 
be an exponent matrix, and
$c \in \R^m_>$ 
be a positive parameter vector.
They define the
parametrized system of (strict or non-strict) generalized polynomial inequalities
\begin{align*}
\sum_{j=1}^m a'_{ij} \, c_j \, x_1^{b_{1j}} \cdots x_n^{b_{nj}} > 0 , \quad i=1,\ldots,\rka' , \\
\sum_{j=1}^m a''_{ij} \, c_j \, x_1^{b_{1j}} \cdots x_n^{b_{nj}} \ge 0 , \quad i=1,\ldots,\rka'' 
\end{align*}
in $n$ positive variables $x_i>0$, $i=1,\ldots,n$,
and involving $m$ monomials $x_1^{b_{1j}} \cdots x_n^{b_{nj}}$, $j=1,\ldots,m$.
(We allow $\rka' = 0$, that is, only non-strict inequalities, 
and analogously $\rka'' = 0$, that is, only strict inequalities.)
In compact form, 
\begin{equation} \label{eq:problemAB}
A' \left( c \circ x^B \right) > 0 , \quad 
A'' \left( c \circ x^B \right) \ge 0
\end{equation}
for $x \in \R^n_>$.

We obtain~\eqref{eq:problemAB} as follows.
From the exponent matrix $B = (b^1,\ldots,b^m)$,
we define the monomials $x^{b^j} = x_1^{b_{1j}} \cdots x_n^{b_{nj}} \in \R_>$,
the vector of monomials $x^B \in \R^m_>$ via $(x^B)_j = x^{b^j}$,
and the vector of monomial terms $c \circ x^B \in \R^m_>$ using the componentwise product~$\circ$.
(All notation is formally introduced at the end of this introduction.)

Clearly, the vector of monomial terms $(c \circ x^B) \in \R^m_>$ is positive.
If~\eqref{eq:problemAB} holds, 
it lies in the polyhedral cones
$C' = \{ y \in \R^m_\ge \mid A' \, y \ge 0 \}$ and $C'' = \{ y \in \R^m_\ge \mid A'' \, y \ge 0 \}$,
more specifically, in the positive parts of $\relint C'$ and $C''$.
Indeed, the crucial object is the convex cone 
\begin{equation*}
C = \{ y \in \R^m_> \mid A' \, y > 0, \, A'' \, y \ge 0 \},
\end{equation*}
a polyhedral cone with some faces removed.
It allows us to write system~\eqref{eq:problemAB} as
\begin{equation} \label{eq:problemBC}
\left(c \circ x^B\right) \in C .
\end{equation}
In this work, we start from an arbitrary cone $C \subseteq \R^m_>$ in the positive orthant,
call it the {\em coefficient} cone,
and refer to \eqref{eq:problemBC} as a {\em parametrized system of generalized polynomial inequalities}
(for given $B$ and $C$).

On the one hand, 
system~\eqref{eq:problemAB}
encompasses parametrized systems of generalized polynomial {\em equations},
\begin{equation*}
A \left( c \circ x^B \right) = 0
\end{equation*}
with $A \in \R^{\rka \times m}$,
which allows for applications in two areas:
(i)~fewnomial systems, see e.g.\ \cite{Khovanskii1991,Sottile2011}, and
(ii)~reaction networks with (generalized) mass-action kinetics,
see e.g.\ \cite{HornJackson1972,Horn1972,Feinberg1972} and \cite{MuellerRegensburger2012,MuellerRegensburger2014,Mueller2016,MuellerHofbauerRegensburger2019}.
We depict a hierarchy of systems in Figure~\ref{fig:1},
ranging from system~\eqref{eq:problemBC} to fewnomial and generalized mass-action systems.

On the other hand, 
system~\eqref{eq:problemBC} 
allows for finitely or infinitely many, strict or non-strict inequalities
and hence for another area of application:
(iii) semi-algebraic sets~\cite{Bochnak1998,Basu2003} with positivity conditions,
that is, finite unions of sets
given by equations ${A \, x^B = 0}$ and strict inequalities ${A' \, x^B > 0}$ 
with $x \in \R^n_>$, $A \in \R^{\rka \times m}$, $A' \in \R^{\rka' \times m}$, 
and $B \in \N_0^{n \times m}$ (over integers, rather than over reals).
For a survey on effective quantifier elimination including applications with positivity conditions,
see~\cite{Sturm2017},
and for the existence of positive solutions to a class of parametrized systems of polynomial inequalities,
see~\cite{HongSturm2018}.

\begin{figure}[t]
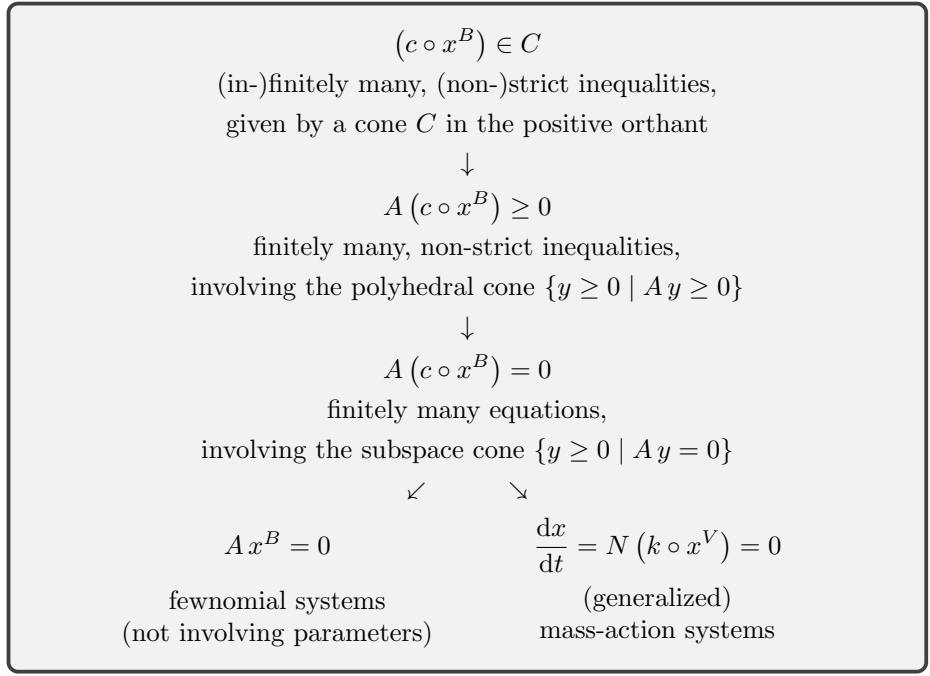

\begin{tcolorbox}
\vspace{-3ex}
\begin{gather*}
\left(c \circ x^B\right) \in C \\ 
\text{(in-)finitely many, (non-)strict inequalities,} \\ \text{given by a cone $C$ in the positive orthant} 
\\
\downarrow \\
A \left(c \circ x^B\right) \ge 0 \\ 
\text{finitely many, non-strict inequalities,} \\ \text{involving the polyhedral cone $\{ y \ge 0 \mid A \, y \ge 0 \}$} \\
\downarrow \\
A \left(c \circ x^B\right) = 0 \\ 
\text{finitely many equations,} \\ \text{involving the subspace cone $\{ y \ge 0 \mid A \, y = 0 \}$} \\
\rotatebox[origin=c]{-45}{$\downarrow$} \hspace{1cm} \rotatebox[origin=c]{45}{$\downarrow$} \\
\parbox[t]{5cm}{\centering $A \, x^B = 0$ \\[2ex] fewnomial systems \\ (not involving parameters)}
\parbox[t]{5cm}{\centering $\displaystyle \dd{x}{t} = N \left(k \circ x^V\right) = 0$ \\[1ex] (generalized) \\ mass-action systems}
\end{gather*}
\end{tcolorbox}
\label{fig:1}
\caption{A hierarchy of (parametrized) systems of generalized polynomial equations and inequalities for positive variables.} 
\end{figure}

The primary contributions of this work are:
\begin{enumerate}
\item
We identify the relevant geometric objects of system~\eqref{eq:problemBC},
namely, the {\em coefficient set}~$P$ (a bounded set), 
the monomial difference subspace~$\LL$,
and the monomial dependency subspace~$D$.

More specifically,
we consider partitions of the monomials into {\em classes},
corresponding to a decomposition of the coefficient cone (as a direct product),
and we obtain the coefficient set $P = C \cap \Delta$ 
by intersecting the coefficient cone $C$ with a direct product $\Delta$ of simplices (or appropriate affine subspaces) on the classes.

The monomial difference subspace $L$ is determined by the affine monomial spans of the classes,
whereas the monomial dependency subspace $D$ captures affine dependencies within {\em and} between classes.
In particular, the {\em monomial dependency}~$\dep = \dim D$ is crucial.
\item
As our main result,
we rewrite polynomial inequalities in terms of $\dep$ binomial equations on $P$,
involving $\dep$ monomials in the parameters.

In particular, we establish an {\em explicit} bijection between
the solution set \[ Z_c = \{ x \in \R^n_> \mid \left( c \circ x^B \right) \in C \} \] of system~\eqref{eq:problemBC} 
and the solution set on $P$,
 \[ Y_c = \{ y \in P \mid y^z = c^z \text{ for all } z \in D \} , \]
via exponentiation.
\item
We obtain a problem classification.
The dependency $d$ and the dimension of~$P$ indicate the complexity of a system.

If $\dep=0$ (the ``very few''-nomial case), solutions exist (for all parameters) and can be parametrized explicitly.
If $\dep>0$, 
the treatment of the $\dep$~equations requires additional objects 
such as sign-characteristic functions (introduced in this work).
\end{enumerate}

Our results lay the groundwork for a novel approach to ``positive algebraic geometry''.
They are based on methods from linear algebra and convex/polyhedral geometry
(and complemented by techniques from analysis).

In applications to fewnomial systems, we extend the standard setting in three ways:
we allow for generalized polynomial {\em inequalities},
we assign a positive {\em parameter} to every monomial,
and we consider partitions of the monomials into {\em classes}.
In applications to (generalized) mass-action systems,
parameters and classes are standard,
but our {\em classes} differ from the ``linkage classes'' of reaction network theory.
See also Example~\ref{exa:two-component} and the footnote there.

Indeed, we illustrate our results (in particular, the relevant geometric objects) through three examples 
from real fewnomial and reaction network theory,
all of which involve trinomials.
First, we compute an explicit parametrization of the solution set
for a ``very few''-nomial system (with $\dep=0$) arising from a reaction network.
Second, we provide a ``solution formula'' for univariate trinomials (with $\dep=1$).
To this end, we introduce sign-characteristic functions and corresponding discriminants and ``roots''.
Third, 
we consider a system of one trinomial equation and one tetranomial inequality (with $\dep=2$ and two classes)
and provide an explicit parametrization of the region for multistationarity of the underlying reaction network.

We foresee applications of our approach to many more problems
such as 
existence and uniqueness of positive solutions, for given or all parameters, 
upper bounds for the number of solutions/components of fewnomial systems,
and extensions of classical results from reaction network theory. 
For first applications to real fewnomial theory, see the parallel work~\cite{MuellerRegensburger2023b}.
There, we improve upper bounds (on the number of positive solutions)
for (i) $n$ trinomials involving ${n+2}$ monomials in $n$ variables, given in~\cite{Bihan2021},
and (ii) one trinomial and one $t$-nomial (with $t\ge3$) in two variables, given in~\cite{Li2003,Koiran2015a}.
Further, for two trinomials ($t=3$), we refine the known upper bound of five in terms of the exponents.
For a characterization of the existence of a unique solution
(and a resulting multivariate Descartes' rule of signs), see the very recent work~\cite{DeshpandeMueller2024}.

{\bf Organization of the work.} 
In Section~\ref{sec:help}, 
we formally introduce the geometric objects and auxiliary matrices
required to rewrite system~\eqref{eq:problemBC}. 
In Section~\ref{sec:main}, we state and prove our main results, Theorem~\ref{thm:main} and Proposition~\ref{pro:bij}.
Moreover, 
in~\ref{sec:d0}, we discuss 
explicit parametrizations of the solution set beyond monomial parametrizations,
and 
in~\ref{sec:decomp}, we consider systems that are decomposable into subsystems given by the classes.
Finally, in Section~\ref{sec:app},
we apply our results to three examples,
all of which involve trinomials.
We briefly summarize all examples at the beginning of the section.

In Appendix~\ref{app:sc}, we introduce sign-characteristic functions,
which serve as a key technique in the analysis of trinomials.

\subsection*{Notation}

We denote the positive real numbers by $\R_>$ and the nonnegative real numbers by $\R_\ge$.
We write $x>0$ for $x \in \R^n_>$ and $x \ge 0$ for $x \in \R^n_\ge$.
For $x \in \R^n$, we define $\e^x \in \R^n_>$ componentwise, and for $x \in \R^n_>$, we define $\ln x \in \R^n$.

For $x \in \R^n_>, \, y \in \R^n$, we define the generalized monomial $x^y = \prod_{i=1}^n (x_i)^{y_i} \in \R_>$,
and for $x \in \R^n_>, \, Y = ( y^1 \ldots \, y^m) \in \R^{n \times m}$, 
we define the vector of monomials $x^Y \in \R^m_>$ via $(x^Y)_i =x^{y^i}$. 

For $x,y \in \R^n$, we denote their scalar product by $x \cdot y \in \R$
and their componentwise (Hadamard) product by $x \circ y \in \R^n$,
and for 
\[
\alpha \in \R^\ell , \quad x = \begin{pmatrix} x^1 \\ \vdots \\ x^\ell \end{pmatrix} \in \R^n
\]
with $x^1 \in \R^{n_1}$, \ldots, $x^\ell \in \R^{n_\ell}$ and $n = n_1 + \cdots + n_\ell$,
we introduce
\[
\alpha \odot x = \begin{pmatrix} \alpha_1 \, x^1 \\ \vdots \\ \alpha_\ell \, x^\ell \end{pmatrix} \in \R^n .
\]

We write $\id_n \in \R^{n \times n}$ for the identity matrix and $\ones_n \in \R^n$ for the vector with all entries equal to one.
Further, for $n >1$, we introduce the incidence matrix $\IL_n \in \R^{n \times (n-1)}$ of the (star-shaped) graph 
$n \to 1, \, n \to 2, \, \ldots, \, n \to (n-1)$
with $n$ vertices and $n-1$ edges, 
\[
\IL_n = \begin{pmatrix} \id_{n-1} \\ - \ones_{n-1}^\trans \end{pmatrix} , \quad \text{that is,} \quad
\IL_2 = \begin{pmatrix} 1 \\ - 1 \end{pmatrix} , \;
\IL_3 = \begin{pmatrix} 1 & 0 \\ 0 & 1 \\ -1 & -1 \end{pmatrix} , \text{ etc.}
\]
Clearly, $\ker \IL_n = \{0\}$ and $\ker \IL_n^\trans = \im \ones_n$. 


\section{Geometric objects and auxiliary matrices} \label{sec:help}

From the introduction,
recall the exponent matrix $B \in \R^{n \times m}$,
the positive parameter vector $c \in \R^m_>$,
and the coefficient cone $C \subseteq \R^m_>$. 
We introduce geometric objects and auxiliary matrices
required to rewrite the parametrized system of generalized polynomial inequalities~\eqref{eq:problemBC}.

\begin{subequations}
In the following, we assume that $C$ is not empty (as a necessary condition for the existence of solutions).
Further, we assume that
\begin{equation}
C = 
C_1 \times \cdots \times C_\ell
\end{equation}
with
\begin{equation}
C_i \subseteq \R^{m_i}_> ,
\end{equation}
$m_1+\ldots+m_\ell = m$, and $\ell\ge1$,
that is, $C$ is a direct product of cones.
\end{subequations}

Accordingly, 
\begin{equation}
B = \begin{pmatrix} B_1 & \ldots & B_\ell \end{pmatrix} \in \R^{n \times m}
\end{equation}
with $\ell$ blocks $B_i \in \R^{n \times m_i}$ and 
\begin{equation}
c= \begin{pmatrix} c^1 \\ \vdots \\ c^\ell \end{pmatrix} \in \R^m_>
\end{equation} 
with $c^i \in \R^{m_i}_>$.

The decomposition $C = C_1 \times \cdots \times C_\ell$
induces a partition of the indices $\{1,\ldots,m\}$ into $\ell$ {\em classes}.
Correspondingly, the columns of $B=(b^1,\ldots,b^m)$, 
and hence the monomials $x^{b^j}$, $j=1,\ldots,m$
are partitioned into classes.
Our main results hold for any partition that permits a direct product form of $C$,
but the finest partition allows us to maximally reduce the problem dimension.

Indeed,
going from a cone to a bounded set reduces dimensions by one per class.
Hence, we introduce the direct product
\begin{subequations}
\begin{equation}
\Delta = \Delta_{m_1-1} \times \cdots \times \Delta_{m_\ell-1} 
\end{equation}
of the standard simplices
\begin{equation}
\Delta_{m_i-1} = \{ y \in \R^{m_i}_\ge \mid \ones_{m_i} \cdot y = 1 \} 
\end{equation}
\end{subequations}
and define the bounded set
\begin{subequations}
\begin{equation}
P = C \cap \Delta .
\end{equation}
For computational simplicity in examples, 
we often use sets of the form $\tilde \Delta_{m_i-1} = \{ y \in \R^{m_i}_\ge \mid u \cdot y = 1 \}$ 
for some vector $u \in \relint C_i^*$, where $C_i^*$ is the dual cone of $C_i$.

Clearly,
\begin{equation}
P = P_1 \times \cdots \times P_\ell
\end{equation}
with
\begin{equation}
P_i = C_i \cap \Delta_{m_i-1} .
\end{equation}
We call $P$ the {\em coefficient set}.
\end{subequations}

Now, let
\begin{equation}
\IL = \begin{pmatrix} \IL_{m_1} & & 0 \\ & \ddots & \\ 0 & & \IL_{m_\ell} \end{pmatrix} \in \R^{m \times (m-\ell)}
\end{equation}
be the $\ell \times \ell$ block-diagonal (incidence) matrix
with blocks $\IL_{m_i} \in \R^{m_i \times (m_i-1)}$,
and let
\begin{subequations}
\begin{equation}
M = B \, \IL \in \R^{n\times(m-\ell)} .
\end{equation}
Clearly, $\ker \IL = \{0\}$ and $\ker \IL^\trans = \im \ones_{m_1} \times \cdots \times \im \ones_{m_\ell}$.
The matrix
\begin{equation}
M = \begin{pmatrix} B_1 \IL_{m_1} & \ldots &  B_\ell \IL_{m_\ell} \end{pmatrix}
\end{equation}
\end{subequations}
records the differences of the columns of $B$ within classes,
explicitly, between the first $m_i-1$ columns of $B_i$ and its last column, for all $i=1,\ldots,\ell$. 
Hence,
\begin{equation}
\LL = \im M \subseteq \R^n
\end{equation}
is the sum of the linear subspaces associated with the affine spans of the columns of $B$ in the $\ell$ classes.
We call $\LL$ the {\em monomial difference subspace}.
Further, we call
\begin{equation}
\dep = \dim (\ker M)
\end{equation}
the {\em monomial dependency}.

\begin{pro} \label{pro:dep}
The monomial dependency can be determined as
\[
\dep = m-\ell-\dim \LL .
\]
\end{pro}
\begin{proof}
Clearly, $\dim(\im M) + \dim(\ker M) = m-\ell$, by the rank-nullity theorem for $M$. 
\end{proof}

We call a system {\em generic} if $M \in \R^{n \times (m-\ell)}$ has full (row or column) rank.

\begin{pro} 
A system is generic if and only if $\LL=\R^n$ or $\dep=0$.
\end{pro}
\begin{proof}
Clearly, $M$ has full (row or column) rank
if and only if
$\ker M^\trans = \{0\}$ or $\ker M = \{0\}$.
The former is equivalent to $\LL=\im M = (\ker M^\trans)^\perp = \R^n$,
the latter is equivalent to $\dep=0$.
\end{proof}

\begin{rem}
In real fewnomial theory,
it is standard to assume $\LL = \R^n$. Otherwise, dependent variables can be eliminated.
In reaction network theory, however, non-generic systems are also relevant.
In both settings, a generic system with $\dep=0$ allows for an explicit parametrization of the solution set;
see our main result and its instance, Theorem~\ref{thm:main} and Corollary~\ref{cor:d0}, below.
\end{rem}

To conclude our exposition,
let
\begin{equation}
\JL = \begin{pmatrix} \ones_{m_1}^\trans & & 0 \\ & \ddots & \\ 0 & & \ones_{m_\ell}^\trans \end{pmatrix} \in \R^{\ell \times m}
\end{equation}
be the $\ell \times \ell$ block-diagonal ``Cayley'' matrix
with blocks $\ones_{m_i}^\trans \in \R^{1 \times m_i}$
and
\begin{equation}
\mBp = \begin{pmatrix} B \\ \JL \end{pmatrix} \in \R^{(n+\ell) \times m} .
\end{equation}
Clearly, $\JL \, \IL = 0$, in fact, $\ker \JL = \im \IL$.
We call 
\begin{equation}
D = \ker \mBp 
\subset \R^m
\end{equation}
the {\em monomial dependency subspace}.

\begin{lem} \label{lem:Gale}
Let $\mG \in \R^{(m-\ell) \times \dep}$ represent a basis for $\ker M$, that is, $\im \mG = \ker M$ (and $\ker \mG = \{0\}$),
and let $\mH = \IL \, \mG \in \R^{m \times \dep}$. 
Then $\mH$ represents a basis of $\ker \mBp$,
that is, 
\[
D = \ker \mBp = \im \mH 
\quad (\text{and }
\dim D = \dep) .
\]
\end{lem}
\begin{proof}
We show $\im \mH \subseteq \ker \mBp$ and $\dim(\im \mH) = \dim(\ker \mBp)$.
First,
\[
\mBp \, \mH = \begin{pmatrix} B \\ \JL \end{pmatrix} \IL \, \mG =
\begin{pmatrix} M \\ 0 \end{pmatrix} \mG = 0 .
\]
Second, $\ker \mG=\{0\}$ and $\ker I=\{0\}$. Hence, $\ker \mH=\{0\}$ and
\begin{align*}
\dim(\im \mH) = \dep &= \dim(\ker M) \\
&= \dim(\ker B \, \IL) = \dim(\ker B \cap \im \IL) \\
&= \dim(\ker B \cap \ker \JL) = \dim (\ker \mBp) . 
\end{align*}
\end{proof}

For illustrations of all geometric objects and auxiliary matrices,
we refer the reader to Example~\ref{exa:two-component} (two overlapping trinomials in four variables with ${\dep=0}$), 
Example~\ref{exa:tri} (one trinomial in one variable with $\dep=1$), 
and Example~\ref{exa:multi} (one trinomial equation and one tetranomial {\em inequality} in five variables with $\dep = 2$ and $\ell=2$ {\em classes}),
in Section~\ref{sec:app}.
They can be read before or after the presentation of our main results.


\section{Main results} \label{sec:main}

Using the geometric objects and auxiliary matrices introduced in the previous section,
we state and prove our main results, Theorem~\ref{thm:main} and Proposition~\ref{pro:bij} below.
Moreover, 
in Section~\ref{sec:d0}, we discuss 
explicit parametrizations of the solution set, 
and 
in~\ref{sec:decomp}, we consider systems that are decomposable into subsystems given by the classes.

\begin{thm} \label{thm:main}
Consider the parametrized system of generalized polynomial inequalities $(c \circ x^B) \in C$ for the positive variables $x \in \R^n_>$,
given by a real exponent matrix $B \in \R^{n \times m}$, 
a positive parameter vector $c \in \R^m_>$,
and a coefficient cone $C \subseteq \R^m_>$.
The solution set $Z_c = \{ x \in \R^n_> \mid (c \circ x^B) \in C \}$ 
can be written as
\[
Z_c = \{ (y \circ c^{-1})^E \mid y \in Y_c \} \circ \e^{\LL^\perp}
\quad \text{with} \quad
Y_c = \{ y \in P \mid y^z = c^z \text{ for all } z \in D \} .
\]
Here, $P$ is the coefficient set, $D$ is the monomial dependency subspace, $\LL$ is the monomial difference subspace,
and the matrix $E = \IL \, M^*$ is given by the (incidence) matrix $\IL$ and a generalized inverse $M^*$ of $M = B \, \IL$.
\end{thm}
\begin{proof}
We reformulate the generalized polynomial inequalities for $x \in \R^n_>$,
\begin{equation*} \label{eq:prob} \tag{I}
\left(c \circ x^B\right) \in C ,
\end{equation*}
in a series of equivalent problems:
\begin{align*}
c \circ x^B = \bar y, & \quad \bar y \in C , \\
c \circ x^B = \alpha \odot y, & \quad \alpha \in \R^\ell_>, \, y \in P , \\
x^B = \alpha \odot (y \circ c^{-1}), & \quad \alpha \in \R^\ell_>, \, y \in P , \\
B^\trans \ln x = \ln \alpha \odot \ones_m + \ln (y \circ c^{-1}), & \quad \alpha \in \R^\ell_>, \, y \in P , \\
\underbrace{\IL^\trans B^\trans}_{M^\trans} \ln x = \IL^\trans \ln (y \circ c^{-1}), & \quad y \in P . \label{eq:last} \tag{II}
\end{align*}
In the last two steps, we applied the logarithm and multiplied with $\IL^\trans$.

As in Lemma~\ref{lem:Gale},
let $\mG \in \R^{(m-\ell) \times \dep}$ represent a basis of $\ker M$, in particular, $\im \mG = \ker M$,
and let $\mH = \IL \, \mG \in \R^{m \times \dep}$. 
We reformulate the solvability of the linear equations~\eqref{eq:last} for $\ln x \in \R^n$ 
in a series of equivalent problems:
\begin{gather*}
\IL^\trans \ln (y \circ c^{-1}) \in \im M^\trans = \ker \mG^\trans , \\
\underbrace{\mG^\trans \IL^\trans}_{\mH^\trans} \ln (y \circ c^{-1}) = 0 , \\
(y \circ c^{-1})^{\mH} = \ones_\dep , \\
y^{\mH} = c^{\mH} . \label{eq:solve} \tag{III}
\end{gather*}
In the last two steps, we applied exponentiation and multiplied out.

If \eqref{eq:last} has a solution,
that is, if \eqref{eq:solve} holds,
it remains to determine a particular solution (of the inhomogeneous system)
and the solutions of the homogeneous system.

Let $M^* \in \R^{(m-\ell) \times n}$ be a generalized inverse of $M$,
that is, $M M^* M = M$,
and $E = \IL \, M^* \in \R^{m \times n}$.
If \eqref{eq:solve} holds, \eqref{eq:last} is equivalent to
\begin{gather*}
\ln x = \ln x^* + \ln x' 
\intertext{with}
\ln x^* = \underbrace{(M^*)^\trans \IL^\trans}_{E^\trans} \ln (y \circ c^{-1}) 
\quad\text{and}\quad
\ln x' \in \ker M^\trans = \LL^\perp . 
\end{gather*}

After exponentiation, \eqref{eq:last} and hence the original problem~\eqref{eq:prob} are equivalent to
\[
x = (y \circ c^{-1})^E  \circ \e^v, \, v \in {\LL^\perp}, \, y^{\mH} = c^{\mH}, \, y \in P .
\]
Finally, by Lemma~\ref{lem:Gale}, $D = \im \mH$,
and hence $y^{\mH} = c^{\mH}$ is equivalent to 
$y^z = c^z$ for all $z \in D$.
\end{proof}

Theorem~\ref{thm:main} can be read as follows:
In order to determine the solution set $Z_c = \{ x \in \R^n_> \mid (c \circ x^B) \in C \}$,
first determine the {\em solution set on the coefficient set},
\begin{equation} 
Y_c = \{ y \in P \mid y^z = c^z \text{ for all } z \in D \} .
\end{equation}
The coefficient set $P$ is determined by the coefficient cone~$C$ (and its classes),
and the dependency subspace $D$ is determined by the exponent matrix~$B$ ({\em and} the classes of $C$).
Explicitly, using $D = \im \mH$ with $\mH \in \R^{m \times \dep}$, there are $\dep$ binomial equations
\begin{equation} 
y^{\mH} = c^{\mH} , 
\end{equation}
which depend on the parameters via $\dep$ monomials $c^{\mH}$.

To a solution $y \in Y_c$ on the coefficient set,
there corresponds the actual solution $x = (y \circ c^{-1})^E \in Z_c$.
In fact, if (and only if) $\dim L < n$, 
then $y \in Y_c$ corresponds to an exponential manifold of solutions, $x \circ \e^{L^\perp} \subset Z_c$.
By Proposition~\ref{pro:bij} below, 
the set $Z_c / e^{\LL^\perp}$, that is, the equivalence classes of $Z_c$ (given by $x' \sim x$ if $x'=x \circ \e^v$ for some $v \in L^\perp$),
and the set $Y_c$ are in one-to-one correspondence.
\begin{pro} \label{pro:bij}
There is a bijection between $Z_c / e^{\LL^\perp}$ and $Y_c$.
\end{pro}
\begin{proof}
By Theorem~\ref{thm:main},
to every $y \in Y_c$, there corresponds the set $x \circ e^{\LL^\perp} \subseteq Z_c$ with $x = (y \circ c^{-1})^E \in Z_c$.
Conversely, 
by the problem definition $c \circ x^B \in C$,
to every $x \in Z_c$, there corresponds a unique $\bar y = (c \circ x^B) \in C$
and, after intersecting the product of rays $\rho = \{\alpha \odot \bar y \mid \alpha \in \R_>^\ell\}$ with $\Delta$ (the direct product of $\ell$ simplices or appropriate affine subspaces), 
a unique $y = (\rho \cap \Delta) \in Y_c$.
\end{proof}
%
%

An important special case arises if $\dep=0$ (the ``very few''-nomial case) and hence $Y_c = P$.
Then the solution set $Z_c$ in Theorem~\ref{thm:main}
has an explicit parametrization.

\begin{cor} \label{cor:d0}
If $\dep=0$, 
then
\[
Z_c = \{ (y \circ c^{-1})^E \mid y \in P \} \circ \e^{\LL^\perp} ,
\]
in particular, the parametrization is explicit,
and there exists a solution for all~$c$.
The solution is unique
if and only if $\dim P = 0$ and $\dim \LL^\perp = 0$.
\end{cor}


\subsection{\texorpdfstring{Monomial dependency zero 
\\ and monomial parametrizations}{}} 
\label{sec:d0}

If the monomial dependency is zero ($\dep=0$) {\em and} the coefficient set consists of a point ($P=\{ y \}$), 
then the solution set has an exponential/monomial parametrization,
\begin{equation}
Z_c = x \circ e^{\LL^\perp}
\quad\text{with}\quad
x = (y \circ c^{-1})^E ,
\end{equation}
involving the original parameters $c$ via the particular solution $x$,
by~Corollary~\ref{cor:d0}.
It is well known that such sets are solutions to binomial equations.
This can be seen already from the coefficient set/cone.
Indeed, 
\begin{subequations}
\begin{equation}
y = \begin{pmatrix} y^{1} \\ \vdots \\ y^{\ell} \end{pmatrix} \in \R^m_>
\end{equation}
with blocks $y^{i} \in \R^{m_i}_>$ on the $\ell$ classes, that is,
\begin{equation}
P = \{ y^{1} \} \times \cdots \times \{ y^{\ell} \} ,
\end{equation}
and
\begin{equation}
C= \rho_1 \times \cdots \times \rho_\ell
\end{equation}
\end{subequations}
is a direct product of the rays $\rho_i = \{ \alpha \, \bar y^{i} \mid \alpha>0 \} \subset \R^{m_i}_>$.
Hence,
\begin{equation}
C= \ker A \cap \R^m_>
\end{equation}
with a block-diagonal matrix
\begin{subequations}
\begin{equation}
A = \begin{pmatrix} A_1 & & 0 \\ & \ddots & \\ 0 & & A_\ell \end{pmatrix} \in \R^{(m-\ell)\times m}
\end{equation}
with blocks
\begin{equation}
A_i = \IL_{m_i}^\trans \diag((y^{i})^{-1}) \in \R^{(m_i-1)\times m_i} ,
\end{equation}
having ``binomial rows''.
In compact form, 
\begin{equation}
A = \IL^\trans \diag(y^{-1}) ,
\end{equation}
\end{subequations}
and hence the system is given by the binomial equations
\begin{equation} \label{eq:beyond}
A \left( c \circ x^B \right) = \IL^\trans \diag(c \circ y^{-1}) \, x^B = 0 ,
\end{equation}
involving combined parameters/coefficients $c \circ y^{-1}$.

\subsubsection*{Beyond monomial parametrizations} 

If $\dep=0$ and $\dim P>0$, 
then the solution set still has an explicit parametrization, 
by~Corollary~\ref{cor:d0}.
Indeed, it is the component-wise product of 
powers of the coefficient set divided by the parameters
and 
the exponentiation of the orthogonal complement of the monomial difference subspace.

Strictly speaking, such sets are not solutions to binomial equations.
However, 
they can be seen as solutions to binomial equations that are parametrized by the coefficient set $P$. 
In this view, the vector $y$ in Equation~\eqref{eq:beyond} is not fixed, but varies over all $y \in P$.

For an illustration, see Example~\ref{exa:two-component} in Section~\ref{sec:app}.


\subsection{Decomposable systems} \label{sec:decomp}

Recall that we consider partitions of the column indices (and hence of the monomials) into $\ell \ge 1$ classes
such that $C$ (and hence $P$) are direct products.
The system
\begin{equation*}
\left(c \circ x^B\right) \in C
\end{equation*}
with
\begin{equation*}
C = C_1 \times \cdots \times C_\ell , \quad
B = \begin{pmatrix} B_1 & \ldots & B_\ell \end{pmatrix} , \quad \text{and} \quad
c = \begin{pmatrix} c^1 \\ \vdots \\ c^\ell \end{pmatrix} 
\end{equation*}
is determined by the subsystems
\begin{equation*}
\left(c^i \circ x^{B_i}\right) \in C_i , \quad i=1,\ldots,\ell ,
\end{equation*}
with solution sets
\begin{equation}
Z_{c,i} = \{ x \in \R^n_> \mid (c^i \circ x^{B_i}) \in C_i \} , \quad i=1,\ldots,\ell .
\end{equation}
By construction, we have the following result.
\begin{fac}
Let $Z_c = \{ x \in \R^n_> \mid (c \circ x^B) \in C \}$.
Then,
\[
Z_c = Z_{c,1} \cap \cdots \cap Z_{c,\ell} .
\]
\end{fac}

The subsystems are not independent, but linked by $x \in \R^n_>$,
which is reflected by the solution set on the coefficient set.
Let 
\begin{equation} 
D_i = \ker \mBp_i 
\end{equation} 
and 
\begin{equation} 
Y_{c,i} = \{ y \in P_i \mid y^z = (c^i)^z \text{ for all } z \in D_i \}, \quad i=1,\ldots,\ell,
\end{equation}
be the solution sets on the coefficient sets of the subsystems.
Clearly, $D_1 \times \cdots \times D_\ell \subseteq D$.
By construction, we have the following result.
\begin{fac}
Let $Y_c = \{ y \in P \mid y^z = c^z \text{ for all } z \in D \}$.
Then,
\[
Y_c \subseteq Y_{c,1} \times \cdots \times Y_{c,\ell} .
\]
\end{fac}
In particular, $Y_c \neq \emptyset$ implies $Y_{c,i} \neq \emptyset$, $i=1,\ldots,\ell$,
but not vice versa.

Hence, we call a system $(c \circ x^B) \in C$ {\em decomposable} 
if $D$ 
is a direct product with the same classes as $C$,
that is, $D = D_1 \times \cdots \times D_\ell$. 
(Trivially, a system with $\ell=1$ is ``decomposable''.)
Let $\dep_i = \dim D_i$ denote the monomial dependency of the subsystem ${(c^i \circ x^{B_i}) \in C_i}$.
By definition, we have the following result.
\begin{fac}
A system $(c \circ x^B) \in C$ is decomposable if and only if $\dep = \dep_1 + \ldots + \dep_\ell$.
\end{fac}

Most importantly,
the solution set on the coefficient set is a direct product
if and only if the system is decomposable.
We provide a formal statement.
\begin{pro}
Let $Y_c$ be the solution set on the coefficient set
and $Y_{c,i}$, $i=1,\ldots,\ell$, be the solution sets on the coefficient sets of the subsystems.
Then,
\[
Y_c = Y_{c,1} \times \cdots \times Y_{c,\ell}
\quad \text{if and only if} \quad
D = D_1 \times \cdots \times D_\ell .
\]
\end{pro}
In particular, if $\dep = \dep_1 + \ldots + \dep_\ell$, then
$Y_c \neq \emptyset$ if and only if $Y_{c,i} \neq \emptyset$, $i=1,\ldots,\ell$.
Given parametrizations of $Y_{c,i}$ 
and hence of $Y_c$,
our main result, Theorem~\ref{thm:main}, provides a parametrization of the solution set $Z_c$.

\begin{rem}
In the special case of polynomial {\em equations} arising from reaction networks,
a decomposition into subnetworks is called ``independent''
if the stoichiometric subspace is a direct sum of the subnetwork subspaces~\cite{Feinberg1987}.
However,
as discussed above, such subnetworks are truly independent
only if the network is decomposable.
To be more explicit, let $N \in \R^{n \times r}$ be the stoichiometric matrix of a reaction network.
An ``independent'' decomposition of $\im N$ as a direct sum
corresponds to a decomposition of $\ker N$ (and hence of the coefficient cone ${C = \ker N \cap \R^r_>}$) as a direct product.
In our approach,
we always assume that $C$ is a direct product,
but this does not imply that the system is decomposable.

In reaction network theory, 
``independent'' decompositions
have been used in the deficiency-one theorem~\cite{Feinberg1987,Boros2012}.
Additionally, it is assumed that the classes are determined by the connected components (``linkage classes'') of the underlying network.
However, in general, classes are not determined by the connected components,
but by the ``independent'' decompositions of the stoichiometric matrix.
For a corresponding algorithm, see Hernandez et al.~\cite{Hernandez2021}. 
In their recent work~\cite{Hernandez2023}, they study two examples and parametrize positive equilibria 
via ``independent'' decomposition, network translation~\cite{Johnston2014,TonelloJohnston2018}, 
and the theory of generalized mass-action systems~\cite{MuellerRegensburger2012,MuellerRegensburger2014,MuellerHofbauerRegensburger2019,CraciunMuellerPanteaYu2019,JohnstonMuellerPantea2019}.
In our approach, these examples have monomial dependency $\dep=0$
and hence an explicit parametrization. 
In future work, we will revisit and extend the deficiency-one theorem. 
\end{rem}



\section{Examples} \label{sec:app}

We demonstrate our results in three examples,
all of which involve trinomials.
We choose them to cover several ``dimensions'' of the problem.
In particular, we consider examples 
\begin{itemize}
\item 
with monomial dependency $\dep=0$, $\dep=1$, or $\dep = 2$, 
\item
with $\ell=1$ or $\ell=2$ classes,
\item
with or without inequalities,
\item
with finitely or infinitely many solutions,  
and
\item
from real fewnomial or reaction network theory.
\end{itemize}

Ordered by the dependency,
we study the following parametrized systems of (generalized) polynomial inequalities:
\begin{itemize}
\item $\dep=0$:

Example~\ref{exa:two-component}: 
{\bf Two overlapping trinomials involving four monomials in four variables}
(from a reaction network model of a two-component regulatory system in the field of molecular biology)

We determine an explicit parametrization of the infinitely many solutions.
\item $\dep=1$:

Example~\ref{exa:tri}: 
{\bf One trinomial in one variable} (univariate case)

We derive Descartes'/Laguerre's rule of signs
and provide a ``solution formula'' involving discriminants
and ``roots'' (as for a quadratic equation).
\item $\dep = 2$:

Example~\ref{exa:multi}: 
{\bf One trinomial equation and one tetranomial inequality in five variables}
(from a reaction network and the study of its region for multistationarity)

We determine an explicit parametrization of the infinitely many solutions
(and hence of the region for multistationarity).
\end{itemize}

For more examples from real fewnomial theory, see the parallel work~\cite{MuellerRegensburger2023b}.


\subsection{Dependency zero} 

\begin{exa} \label{exa:two-component}
We consider {\bf two overlapping trinomials involving four monomials in four variables} $x,x_p,y,y_p$, namely
\begin{alignat*}{3}
-k_1 \, x &+ k_2 \, x_p y - k_3 \, x y_p &&= 0 , \\
& - k_2 \, x_p y + k_3 \, x y_p + k_4 \, y_p &&=0 , 
\end{alignat*}
which arise from a model of a two-component regulatory system\footnote{
A model of a two-component signaling system is given by the chemical reaction network
$\ce{X} \to \ce{X}_p , \; \ce{X}_p + \ce{Y} \rightleftarrows \ce{X} + \ce{Y}_p , \; \ce{Y}_p \to \ce{Y}$.
A histidine kinase~$\ce{X}$ auto-phosphorylates and transfers the phosphate group to a response regulator~$\ce{Y}$ 
which auto-dephosphorylates, see~e.g.~\cite{Conradi2017}.
\par
Under the assumption of mass-action kinetics,
the associated dynamical system (for the concentrations of the chemical species) is given by
\begin{alignat*}{3}
\dd{x}{t} &= -k_1 \, x && + k_2 \, x_p y - k_3 \, x y_p , \\
\dd{y}{t} &= && - k_2 \, x_p y + k_3 \, x y_p + k_4 \, y_p , 
\end{alignat*}
and the conservation laws $x + x_p = \bar x$ and $y + y_p = \bar y$.
\par
The network has 6~vertices (``complexes''), 3~connected components (``linkage classes''), 2~(dynamically) independent variables,
and hence ``deficiency'' $\delta=6-3-2=1$.
\par
The corresponding polynomial equation system has 4~monomials, 1~class, 3~(``monomially'') independent variables, and hence dependency $\dep=4-1-3=0$.
}
in the field of molecular biology.

Equivalently, $A \, (c\circ x^B) = 0$ with
\[
A = \begin{pmatrix} -1 & 1 & -1 & 0 \\ 0 & -1 & 1 & 1 \end{pmatrix} , \quad
B = \bordermatrix{& \gray{1} & \gray{2} & \gray{3} & \gray{4} \cr \gray{x} & 1 & 0 & 1 & 0 \cr \gray{x_p} & 0 & 1 & 0 & 0 \cr \gray{y} & 0 & 1 & 0 & 0 \cr \gray{y_p} & 0 & 0 & 1 & 1} ,
\quad
c = \begin{pmatrix} k_1 \\ k_2 \\ k_3 \\ k_4 \end{pmatrix} .
\]

In terms of inequalities, 
$(c\circ x^B) \in C$ with coefficient cone $C = \ker A \cap \R^4_>$.
Clearly, $m=n=4$ and $\ell=1$. 

Crucial coefficient data are
\begin{gather*}
\ker A = \im \begin{pmatrix} 0 & 1 \\ 1 & 1 \\ 1 & 0 \\ 0 & 1 \end{pmatrix} , \quad
C = \{
\lambda \begin{pmatrix} 0 \\ 1 \\ 1 \\ 0 \end{pmatrix} + \mu \begin{pmatrix} 1 \\ 1 \\ 0 \\ 1 \end{pmatrix}
\mid \lambda,\mu > 0 \} , \\
\text{and} \quad P = C \cap \Delta_3 = \{ 
\begin{pmatrix} 1-\lambda \\ 1 \\ \lambda \\ 1-\lambda \end{pmatrix}
\mid \lambda \in (0,1) \} ,
\end{gather*}
where we use the (non-standard) simplex $\Delta_3 = \{ y \in \R^4_\ge \mid (1,1,2,1) \, y = 3 \}$.

Crucial exponent data are
\[
M = B \, \IL_4 = \begin{pmatrix} 1 & 0 & 1 \\ 0 & 1 & 0 \\ 0 & 1 & 0 \\ -1 & -1 & 0 \end{pmatrix}
\quad \text{and} \quad
\LL^\perp = \ker M^\trans  
= \im \begin{pmatrix} 0 \\ 1 \\ -1 \\ 0 \end{pmatrix} .
\]
In particular, $\dep = \dim (\ker M) = 0$ and $\dim \LL = \dim (\im M) = 3$.
Alternatively, $\dep = m - \ell - \dim \LL = 4 - 1 - 3 = 0$, by Proposition~\ref{pro:dep}.

It remains to determine a generalized inverse of~$M$ and the resulting exponentiation matrix~$E$.
Indeed, we choose
\[
M^* = \begin{pmatrix} 0 & -1 & 0 & -1 \\ 0 & 1 & 0 & 0 \\ 1 & 1 & 0 & 1 \end{pmatrix}
\quad \text{and} \quad
E = \IL_4 \, M^* = \begin{pmatrix} 0 & -1 & 0 & -1 \\ 0 & 1 & 0 & 0 \\ 1 & 1 & 0 & 1 \\ -1 & -1 & 0 & 0 \end{pmatrix}.
\]
By Theorem~\ref{thm:main} (or Corollary~\ref{cor:d0}), we have an explicit parametrization of
the solution set,
\begin{align*}
Z_c &= \{ (y \circ c^{-1})^E  \mid y \in P \} \circ \e^{\LL^\perp} \\
&= \{ \begin{pmatrix} 
(\frac{\lambda}{k_3})^1 (\frac{1-\lambda}{k_4})^{-1} \\ 
(\frac{1-\lambda}{k_1})^{-1} (\frac{1}{k_2})^1 (\frac{\lambda}{k_3})^1 (\frac{1-\lambda}{k_4})^{-1} \\ 
1 \\ 
(\frac{1-\lambda}{k_1})^{-1} (\frac{\lambda}{k_3})^1 
\end{pmatrix} 
\mid \lambda \in (0,1) \} 
\circ
\exp \, ( \, \im \begin{pmatrix} 0 \\ 1 \\ -1 \\ 0 \end{pmatrix} ) \\
&= \{ \begin{pmatrix} \frac{k_4}{k_3} \frac{\lambda}{1-\lambda} \\ \frac{k_1 k_4}{k_2 k_3} \frac{\lambda}{(1-\lambda)^2} \\ 1 \\ \frac{k_1}{k_3} \frac{\lambda}{1-\lambda} \end{pmatrix} 
\circ
\begin{pmatrix} 0 \\ \tau \\ \frac{1}{\tau} \\ 0 \end{pmatrix} \mid \lambda \in (0,1), \, \tau > 0 \} .
\end{align*}
The parametrization involves the original parameters $c=(k_1,k_2,k_3,k_4)^\trans$, the parametrization of the coefficient set involving $\lambda \in (0,1)$,
and the exponential/monomial parametrization involving $\tau > 0$.

\begin{rem}
An equivalent parametrization (involving $\sigma = k_3 \frac{1-\lambda}{\lambda} \in \R_>$ instead of $\lambda \in (0,1)$) has been obtained 
in \cite[Example~18]{JohnstonMuellerPantea2019},
using network translation~\cite{Johnston2014,TonelloJohnston2018} and the theory of generalized mass-action systems~\cite{MuellerRegensburger2012,MuellerRegensburger2014,MuellerHofbauerRegensburger2019,CraciunMuellerPanteaYu2019}.
However, unlike $\lambda$, the parameter $\sigma$ lacks geometric meaning.
It is also unclear under what conditions network translation leads to a generalized mass-action system with zero (effective and kinetic-order) ``deficiencies''.
In our approach, the dependency can be easily computed,
and $\dep=0$ immediately implies the existence of an explicit parametrization.
\end{rem}
\end{exa}


\subsection{Dependency one}

\begin{exa} \label{exa:tri}
We consider {\bf one trinomial in one variable} (a univariate trinomial) in the form
\[
c_1 \, x^{b_1} + c_2 \, x^{b_2} - 1 = 0
\]
with $b_1,b_2\neq0$, $b_1 \neq b_2$, and $c_1,c_2>0$.
That is, we fix the signs of the terms, but do not assume an order on the exponents $b_1,b_2,0$.
To derive Laguerre's rule of signs, we have to consider all possible orders.
If $b_1 \cdot b_2>0$, then the terms have sign pattern $++-$ (if $b_1,b_2>0$) or $-++$ (if $b_1,b_2<0$),
and there is one sign change.
If $b_1 \cdot b_2<0$, then the terms have sign pattern $+-+$,
and there are two sign changes.

Equivalently, $A \, (c\circ x^B)=0$ with
\[
A = \begin{pmatrix} 1 & 1 & -1 \end{pmatrix} , \quad
B = \begin{pmatrix} b_1 & b_2 & 0 \end{pmatrix} ,
\quad
c = \begin{pmatrix} c_1 \\ c_2 \\ 1 \end{pmatrix} .
\]
In terms of inequalities, 
$(c\circ x^B) \in C$ with coefficient cone $C = \ker A \cap \R^3_>$.
Clearly, $m=3$ and $n=\ell=1$. 

Crucial coefficient data are
\begin{gather*}
\ker A = \im \begin{pmatrix} 1 & 0 \\ 0 & 1 \\ 1 & 1 \end{pmatrix} , \quad
C = \{
\lambda \begin{pmatrix} 1 \\ 0 \\ 1 \end{pmatrix} + \mu \begin{pmatrix} 0 \\ 1 \\ 1 \end{pmatrix}
\mid \lambda,\mu \in \R_> \} , \\
\text{and} \quad P = C \cap \Delta_2 = \{ 
\begin{pmatrix} \lambda \\ 1-\lambda \\ 1 \end{pmatrix}
\mid \lambda \in (0,1) \} ,
\end{gather*}
where we use the (non-standard) simplex $\Delta_2 = \{ y \in \R^3_\ge \mid \ones_3 \cdot y = 2 \}$.

Crucial exponent data are
\[
\mBp = \begin{pmatrix} B \\ \ones_3^\trans \end{pmatrix} = \begin{pmatrix} b_1 & b_2 & 0 \\ 1 & 1 & 1 \end{pmatrix} 
\quad\text{and}\quad
D = \ker \mBp = \im z
\]
with
\[
z = \begin{pmatrix} 1 \\ -\frac{b_1}{b_2} \\ \frac{b_1}{b_2}-1 \end{pmatrix} 
= \begin{pmatrix} 1 \\ -b \\ b-1 \end{pmatrix} ,
\quad\text{where}\quad
b=\frac{b_1}{b_2} .
\]
In particular, $\dep = \dim D = 1$.

Now, we can apply our main result, Theorem~\ref{thm:main}.
Most importantly, we determine the solution set on the coefficient set,
that is, the set of all $y \in P$ such that  
%
\begin{equation*}
y^z = c^z ,
\quad \text{that is,} \quad
\lambda^{1} (1-\lambda)^{-b} = c_1^{1} c_2^{-b} =: c^* .
\end{equation*}
Using the {\em sign-characteristic function} $s_{1,-b}$,
we write
\begin{equation} \label{eq:sol}
s_{1,-b}(\lambda) = c^* .
\end{equation}
(For the definition of sign-characteristic functions, see Appendix~\ref{app:sc}.)
%
\begin{itemize}
\item
If $b_1 \cdot b_2>0$ (that is, $b>0$), then $s_{1,-b}$ is strictly monotonically increasing,
and there is exactly one solution to condition~\eqref{eq:sol},
\[
\lambda = r_{1,-b} \left( c^* \right) ,
\]
for every $c^*$ and hence for all parameters $c$.
(Here, $r_{1,-b}$ is the {\em root} of the sign-characteristic function $s_{1,-b}$, see Appendix~\ref{app:sc}.)

This corresponds to the sign patterns $++-$ or $-++$, that is, to one sign change.
\item
If $b_1 \cdot b_2<0$ (that is, $b<0$), then $s_{1,-b}$ has a maximum,
and there are zero, one, or two distinct solutions,
depending on whether the maximum is less, equal, or greater than the right-hand side,
\[
s_{1,-b}^{\max} \lesseqqgtr c^* .
\]
Equality occurs when the solution equals the maximum point. Then, the solution is of order two.
Hence, counted with multiplicity, there are zero or two solutions.

This corresponds to the sign pattern $+-+$, that is, to two sign changes.

Explicitly, if the {\em discriminant} $\mathcal{D}$ is nonnegative,
\[
\mathcal{D} := s_{1,-b}^{\max} - c^*
= \frac{(-b)^{-b}}{(1-b)^{1-b}}
- c_1^{1} c_2^{-b} 
\ge 0 ,
\]
then the two solutions are determined by the roots
\[
\lambda^- = r_{1,-b}^- \left( c^* \right)
\quad\text{and}\quad
\lambda^+ = r_{1,-b}^+ \left( c^* \right) .
\]
That is, we can treat the general trinomial like a quadratic (the paradigmatic trinomial).
If the discriminant $\mathcal{D}$ is negative, then there is no solution,
see also \cite[Proposition~2]{AlbouyFu2014}.
\end{itemize}
Altogether, we have shown Laguerre's rule of signs for trinomials, in our approach.
Moreover, using discriminants and roots, we have provided a solution formula on the coefficient set.

By Theorem~\ref{thm:main}, we obtain 
the solution set in the original variable $x$,
via exponentiation.
In particular,
we use the matrices $\IL_3$, $M = B \, \IL_3$, $M^*$, and $E = \IL_3 \, M^*$.
\end{exa}


\subsection{Dependency two}

\begin{exa} \label{exa:multi}
We consider {\bf one trinomial equation and one tetranomial {\em inequality} in five variables} $x=(x_1,x_2)^\trans$ and $k = (k_1,k_2,k_3)^\trans$,
\begin{alignat*}{3}
k_3 \, x_1^2 x_2 - k_1 \, x_1+ k_2 \, x_2 \phantom{+} & && = 0 , \\
& k_3 \, x_1^2 - 2 k_3 \, x_1 x_2 + k_1 + k_2 && < 0 , 
\end{alignat*}
which arise from the study of a reaction network,\footnote{
The reaction network $\ce{X_1} \rightleftarrows \ce{X_2}$, $2 \ce{X_1} + \ce{X_2} \to 3\ce{X_1}$ is the running example in~\cite{FeliuTelek2023}.
Under the assumption of mass-action kinetics,
the associated dynamical system (for the concentrations of the species $\ce{X_1}, \ce{X_2}$) is given by
\[
\dd{x_1}{t} = f_k(x) = k_3 \, x_1^2 x_2 - k_1 \, x_1+ k_2 \, x_2
\]
and the conservation law $x_1+x_2 = \bar x$, that is, $w \, x = \bar x$ with $w = (1,1)$.

Let $J$ be the Jacobian matrix of the dynamical system for $(x_1,x_2)$
and $M$ be obtained from~$-J$ by replacing the second (dependent) row by $w$.
Then, 
\[
\det M = k_3 \, x_1^2 - 2 k_3 \, x_1 x_2 + k_1 + k_2 .
\]
Now,  let $\mathcal Z_k = \{ x \mid f_k(x) = 0 \}$
and $\mathcal P_{\bar x} = \{ x \mid w \, x = \bar x \}$.
By~\cite[Theorem~1]{Conradi2017},
if $\det M < 0$ for some $x \in \mathcal Z_k \cap \mathcal P_{\bar x}$,
then the parameter pair $(k,\bar x)$ enables multistationarity, that is, $|\mathcal Z_k \cap \mathcal P_{\bar x}| > 1$.
(And $f_k=0$, $\det M < 0$ define the system in the main text.)
If $\det M = 0$, multistationarity {\em may} occur. For simplicity, we do not consider this ``boundary case''.
}
in particular, of its region for multistationarity, see~\cite[Figure~3]{FeliuTelek2023}.

In our notation,
for
\[
\xi = \begin{pmatrix} x \\ k \end{pmatrix}  = \begin{pmatrix} x_1 \\ x_2 \\ k_1 \\ k_2 \\ k_3 \end{pmatrix} \in \R^5_> ,
\]
we have the equation
$A^e \, (c^e \circ \xi^{B^e})=0$ with
\[
A^e = \begin{pmatrix} 1 & -1 & 1 \end{pmatrix} , \quad
B^e = \bordermatrix{ & \gray{1} & \gray{2} & \gray{3} \cr \gray{x_1} & 2 & 1 & 0 \cr \gray{x_2} & 1 & 0 & 1 \cr \gray{k_1} & 0 & 1 & 0 \cr \gray{k_2} & 0 & 0 & 1 \cr \gray{k_3} & 1 & 0 & 0} ,
\quad
c^e = \begin{pmatrix} 1 \\ 1 \\ 1 \end{pmatrix}
\]
and the inequality
$A^i \, (c^i \circ \xi^{B^i}) < 0$ with
\[
A^i = \begin{pmatrix} 1 & -1 & 1 & 1 \end{pmatrix} , \quad
B^i = \bordermatrix{ & \gray{4} & \gray{5} & \gray{6} & \gray{7} \cr \gray{x_1} & 2 & 1 & 0 & 0  \cr \gray{x_2} & 0 & 1 & 0 & 0  \cr \gray{k_1} & 0 & 0 & 1 & 0 \cr \gray{k_2} & 0 & 0 & 0 & 1 \cr \gray{k_3} & 1 & 1 & 0 & 0} ,
\quad
c^i = \begin{pmatrix} 1 \\ 2 \\ 1 \\ 1 \end{pmatrix} ,
\]
where we order monomials first by total degree and then by variable name.

In terms of inequalities, 
we have $(c\circ \xi^B) \in C$,
where the coefficient cone 
\begin{gather*}
C = C^e \times C^i
\intertext{is given by}
C^e = \ker A^e \cap \R^3_>
\quad\text{and}\quad
C^i = \{ y \mid A^i y < 0 \} \cap \R^4_> .
\end{gather*}
Clearly, $n=5$, $\ell=2$, and $m=3+4=7$. 

Explicitly,
coefficient data are
\begin{gather*}
C^e = \{
\lambda_1 \begin{pmatrix} 1 \\ 1 \\ 0 \end{pmatrix} + \lambda_2 \begin{pmatrix} 0 \\ 1 \\ 1 \end{pmatrix}
\mid \lambda_i > 0 \} , \\
C^i = \{
\mu_1 \begin{pmatrix} 1 \\ 1 \\ 0 \\ 0 \end{pmatrix} + \mu_2 \begin{pmatrix} 0 \\ 1 \\ 1 \\ 0 \end{pmatrix} + \mu_3 \begin{pmatrix} 0 \\ 1 \\ 0 \\ 1 \end{pmatrix}
+ \mu_4 \begin{pmatrix} 0 \\ 1 \\ 0 \\ 0 \end{pmatrix} 
\mid \mu_i > 0 \} , 
\end{gather*}
and the resulting coefficient set
\begin{gather*}
P=P^e\times P^i
\intertext{is given by}
P^e = C^e \cap \Delta_2 = \{ 
\begin{pmatrix} \lambda_1 \\ 1 \\ 1-\lambda_1 \end{pmatrix}
\mid 0 < \lambda_1 < 1 \} , \\
P^i = C^i \cap \tilde \Delta_3 = \{ 
\begin{pmatrix} \mu_1 \\ 1 \\ \mu_2 \\ \mu_3 \end{pmatrix}
\mid 0 < \mu_i \text{ and } \sum_{i} \mu_i < 1 \} ,
\end{gather*}
where we use the (non-standard) simplex $\Delta_2 = \{ y \in \R^3_\ge \mid \ones_3 \cdot y = 2 \}$
and the set $\tilde \Delta_3 = \{ {y \in \R^4_\ge} \mid y_2 = 1 \}$.

Exponent data are
\begin{gather*}
B = \begin{pmatrix} B^e & B^i \end{pmatrix},
\quad
\mBp = \begin{pmatrix} B^e & B^i \\ \ones_3^\trans & 0 \\ 0 & \ones_4^\trans \end{pmatrix} 
= \begin{pmatrix} 
2 & 1 & 0 & 2 & 1 & 0 & 0 \\
1 & 0 & 1 & 0 & 1 & 0 & 0 \\
0 & 1 & 0 & 0 & 0 & 1 & 0 \\
0 & 0 & 1 & 0 & 0 & 0 & 1 \\
1 & 0 & 0 & 1 & 1 & 0 & 0 \\
1 & 1 & 1 & 0 & 0 & 0 & 0 \\ 
0 & 0 & 0 & 1 & 1 & 1 & 1
\end{pmatrix} , \quad \text{and} \\
D = \ker \mBp = \im \mH
\quad\text{with}\quad
\mH = 
\begin{pmatrix} 1 & -1 & 0 & 0 & -1 & 1 & 0 \\ 1 & 0 & -1 & -1 & 0 & 0 & 1 \end{pmatrix}^\trans .
\end{gather*}
In particular, $\dep = \dim D = 2$.

Finally, there are five variables, but no symbolic parameters.
Indeed, we have chosen constant ``parameters'',
\[
c = c^e \times c^i = \begin{pmatrix} 1 & 1 & 1 \end{pmatrix}^\trans \times \begin{pmatrix} 1 & 2 & 1 & 1 \end{pmatrix}^\trans
\]
such that $A^i \in \{-1,1\}^{1 \times 4}$,
and hence $C^i$ and $P^i$ have the simple parametrizations given above.

Now, we can apply our main result, Theorem~\ref{thm:main}.
Most importantly, we determine the solution set on the coefficient set,
that is, the set of all $y \in P$ such that  $y^{\mH} = c^{\mH}$.
Explicitly,
\begin{gather*}
\lambda_1^{1} \, 1^{-1} \, 1^{-1} \, \mu_2^{1} = 2^{-1} , \\
\lambda_1^{1} \, (1-\lambda_1)^{-1} \, \mu_1^{-1} \, \mu_3^{1} = 1 . 
\end{gather*}
That is, $\mu_2 = \frac{1}{2 \lambda_1}$ and $\mu_3 = \mu_1 \frac{1-\lambda_1}{\lambda_1}$,
and we choose $\lambda_1, \mu_1$ as independent parameters.
They have to fulfill the defining conditions of $P^e$ and $P^i$,
\[
0 < \lambda_1 < 1, \quad 
0 < \mu_1, \quad\text{and}\quad
\mu_1 + \frac{1}{2 \lambda_1} + \mu_1 \, \frac{1-\lambda_1}{\lambda_1} < 1 .
\]
Equivalently,
\[
0 < \mu_1
\quad\text{and}\quad
\frac{1}{2} + \mu_1 < \lambda_1 < 1 ,
\]
and the solution set on the coefficient set has an explicit parametrization,
\[
Y_c = \{ (\lambda_1,1,1-\lambda_1,\mu_1,1,\textstyle \frac{1}{2 \lambda_1},\textstyle \mu_1 \frac{1-\lambda_1}{\lambda_1})^\trans \mid 0 < \mu_1 \text{ and } \frac{1}{2} + \mu_1 < \lambda_1 < 1 \} ,
\]
involving a system of {\em linear} inequalities.

By Theorem~\ref{thm:main}, we obtain an explicit parametrization of
the solution set in the original variables $\xi = {x \choose k}$,
\[
Z_c = \{ (y \circ c^{-1})^E \mid y \in Y_c \} \circ \e^{\LL^\perp} ,
\]
via exponentiation.

For completeness,
we provide the matrix $E = \IL \, M^*$
(via the incidence matrix $\IL$, $M = B \, \IL$, and a generalized inverse $M^*$)
and the linear subspace $\LL^\perp = \ker M^\trans$,
\[
\IL =
\begin{pmatrix}
1 & 0 & 0 & 0 & 0 \\
0 & 1 & 0 & 0 & 0 \\
-1 & -1 & 0 & 0 & 0 \\
0 & 0 & 1 & 0 & 0 \\
0 & 0 & 0 & 1 & 0 \\
0 & 0 & 0 & 0 & 1 \\
0 & 0 & -1 & -1 & -1
\end{pmatrix} ,
\quad
M = B \, \IL = 
\begin{pmatrix}
2 & 1 & 2 & 1 & 0 \\
0 & -1 & 0 & 1 & 0 \\
0 & 1 & 0 & 0 & 1 \\
-1 & -1 & -1 & -1 & -1 \\
1 & 0 & 1 & 1 & 0
\end{pmatrix} , 
\]
\[
M^* = 
\begin{pmatrix}
1/2 & -1/2 & 0 & 0 & 0 \\
0 & 0 & 0 & 0 & 0 \\
0 & 0 & 0 & 0 & 0 \\
0 & 1 & 0 & 0 & 0 \\
0 & 0 & 1 & 0 & 0
\end{pmatrix} ,
\quad
E = \IL \, M^* = 
\begin{pmatrix}
1/2 & -1/2 & 0 & 0 & 0 \\
0 & 0 & 0 & 0 & 0 \\
-1/2 & 1/2 & 0 & 0 & 0 \\
0 & 0 & 0 & 0 & 0 \\
0 & 1 & 0 & 0 & 0 \\
0 & 0 & 1 & 0 & 0 \\
0 & -1 & -1 & 0 & 0
\end{pmatrix} , 
\]
and 
\[
\LL^\perp = \im
\begin{pmatrix}
1 & 0 \\
1 & 0 \\
0 & 1 \\
0 & 1 \\
-2 & 1 
\end{pmatrix} .
\]

Hence, 
\[
(y \circ c^{-1})^E 
=
\begin{pmatrix} 
\lambda_1^{1/2} (1-\lambda_1)^{-1/2} \\ 
\lambda_1^{-1/2} (1-\lambda_1)^{1/2} (1/2)^1 (\mu_1 \frac{1-\lambda_1}{\lambda_1})^{-1} \\
(\frac{1}{2\lambda_1})^1 (\mu_1 \frac{1-\lambda_1}{\lambda_1})^{-1} \\
1 \\
1 
\end{pmatrix}
=
\begin{pmatrix} 
(\frac{\lambda_1}{1-\lambda_1})^{1/2} \\ 
(\frac{\lambda_1}{1-\lambda_1})^{1/2} \frac{1}{2\mu_1} \\ 
\frac{1}{1-\lambda_1} \frac{1}{2\mu_1}\\
1 \\
1
\end{pmatrix}
\]
for $y \in Y_c$,
\[
\e^{\LL^\perp}
=
\exp \, ( \, \im
\begin{pmatrix}
1 & 0 \\
1 & 0 \\
0 & 1 \\
0 & 1 \\
-2 & 1 
\end{pmatrix} 
) 
=
\{
\begin{pmatrix}
\tau_1 \\
\tau_1 \\
\tau_2 \\
\tau_2 \\
\frac{\tau_2}{\tau_1^2} 
\end{pmatrix} 
\mid \tau_i > 0
\} ,
\]
and
\[
Z_c = \{ 
\begin{pmatrix} 
(\frac{\lambda_1}{1-\lambda_1})^{1/2} \\ 
(\frac{\lambda_1}{1-\lambda_1})^{1/2} \frac{1}{2\mu_1} \\ 
\frac{1}{1-\lambda_1} \frac{1}{2\mu_1} \\
1 \\
1
\end{pmatrix} 
\circ
\begin{pmatrix}
\tau_1 \\
\tau_1 \\
\tau_2 \\
\tau_2 \\
\frac{\tau_2}{\tau_1^2} 
\end{pmatrix} 
\mid
0 < \mu_1, \, 
\frac{1}{2} + \mu_1 < \lambda_1 < 1, \text{ and }
\tau_i > 0
\} .
\]

The solution set $Z_c$ in the original variables $\xi = (x,k)^\trans$ with $x=(x_1,x_2)^\trans$ and $k=(k_1,k_2,k_3)^\trans$
can be projected to the variables $(k,\bar x)$, where $\bar x = x_1+x_2$.

In fact, we fix $k_2=k_3=1$ and project $Z_c$ onto $k_1$ and $\bar x$.
(The projection is the area between the blue lines.)
\begin{center}
\includegraphics[width=0.3\textwidth]{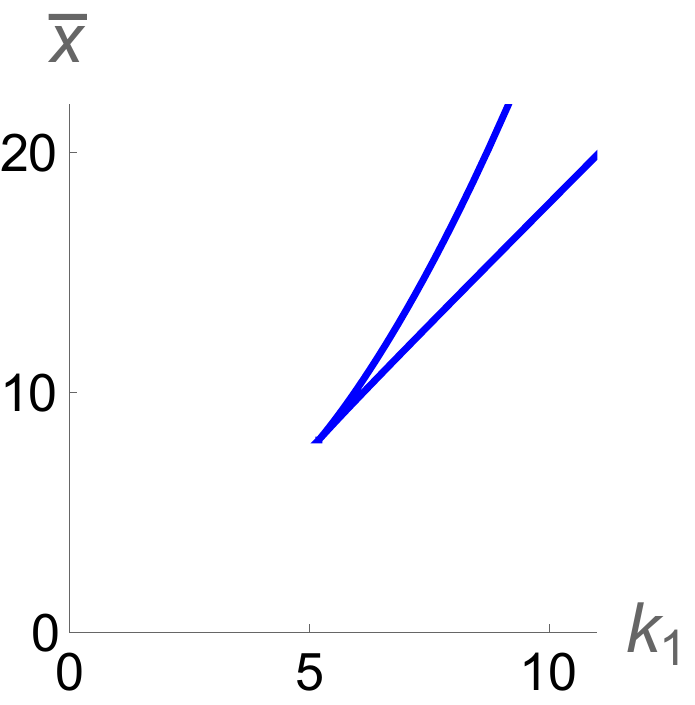}
\end{center}

\begin{rem}
In terms of the underlying reaction network,
the rate constants~$k$ and the total amount $\bar x$ determine the concentrations $x$
(via the equations for steady state and mass conservation).
If, in addition to the steady state equation, 
a certain strict inequality is fulfilled, then the parameter pair $(k,\bar x)$ enables multistationarity,
that is, more than one solution $x$.
(See the footnote at the beginning of the example.)

Above, we have determined the pairs $(x,k)$, with $k$ treated as a variable, that fulfill the equation/inequality, 
that is, the set $Z_c$, which can then be projected onto $(k,\bar x)$.
In fact, we have first determined an {\em explicit} parametrization of $Y_c$,
the solution set on the coefficient set, and then obtained $Z_c$ via exponentiation.
Since $Y_c$ is connected, the same holds for $Z_c$ and, via projection, for the region for multistationarity (given by the strict inequality).

The connectivity of the full region of multistationarity (including ``boundary cases'')
has been established by other methods in~\cite{FeliuTelek2023}, 
notably without determining an explicit parametrization. 
\end{rem}
\end{exa}

\subsection*{Acknowledgements}

We thank Elisenda Feliu and M\'at\'e Telek for comments on Example~\ref{exa:multi}.

This research was funded in whole, or in part, by the Austrian Science Fund (FWF), 
grant DOIs 10.55776/P33218 and 10.55776/PAT3748324 to SM and grant DOI 10.55776/P32301 to GR.

\subsection*{Conflict of interest and data availability}

On behalf of all authors, the corresponding author states 
that there is no conflict of interest
and that the manuscript has no associated data.



\bibliographystyle{abbrv} 
\bibliography{MR,linearstability,polynomials,crnt}


\clearpage

\appendix 
\section*{Appendix}


\section{Sign-characteristic functions} \label{app:sc}

As a key technique for the analysis of trinomials,
we introduce a family of {\em sign-characteristic} functions
on the open unit interval.
For $\alpha,\beta \in \R$, let
\begin{align*}
s_{\alpha,\beta} \colon & (0,1) \to \R_>, \\ 
& \lambda \mapsto \lambda^\alpha (1-\lambda)^\beta .
\end{align*}
For $\alpha,\beta\neq0$,
\[
s'_{\alpha,\beta}(\lambda) = s_{\alpha-1,\beta-1}(\lambda) \left( \alpha(1-\lambda)-\beta \lambda \right) .
\]
Hence,
$s_{\alpha,\beta}$ has an extremum at
\[
\lambda^*= \frac{\alpha}{\alpha+\beta} \in (0,1)
\]
if and only if $\alpha \cdot \beta>0$.
Then, 
\[
s_{\alpha,\beta}(\lambda^*) = \left( \frac{\alpha}{\alpha+\beta} \right)^\alpha \left( \frac{\beta}{\alpha+\beta} \right)^\beta .
\]
We call the functions sign-characteristic since the signs ($-,0$ or $+$) of $\alpha$ and $\beta$
characterize the values ($0,1$ or $\infty$) at $0+$ and $1-$, respectively.
In particular,
\begin{itemize}
\item
if $\alpha,\beta>0$, then $s_{\alpha,\beta}(0+)=s_{\alpha,\beta}(1-)=0$, and $s_{\alpha,\beta}$ has a maximum at~$\lambda^*$,
\item
if $\alpha,\beta<0$, then $s_{\alpha,\beta}(0+)=s_{\alpha,\beta}(1-)=\infty$, and $s_{\alpha,\beta}$ has a minimum at~$\lambda^*$,
\item
if $\alpha>0>\beta$, then $s_{\alpha,\beta}(0+)=0, \, s_{\alpha,\beta}(1-)=\infty$, and $s_{\alpha,\beta}$ is strictly monotonically increasing, and
\item
if $\alpha<0<\beta$, then $s_{\alpha,\beta}(0+)=\infty, \, s_{\alpha,\beta}(1-)=0$, and $s_{\alpha,\beta}$ is strictly monotonically decreasing.
\end{itemize}

\subsection*{Roots of sign-characteristic functions}

On the monotonic parts of the sign-characteristic functions,
we introduce inverses or ``roots''.
In particular,
\begin{itemize}
\item
if $\alpha \cdot \beta<0$, then we define
\begin{align*}
r_{\alpha,\beta} \colon & \R_> \to (0,1) , \\
& \lambda \mapsto (s_{\alpha,\beta})^{-1}(\lambda) .
\end{align*}
\item
If $\alpha \cdot \beta>0$, then we denote the restrictions of $s_{\alpha,\beta}$ to $(0,\lambda^*]$ and $[\lambda^*,1)$
by $s_{\alpha,\beta}^-$ and $s_{\alpha,\beta}^+$, respectively, and we define
\begin{alignat*}{4}
r_{\alpha,\beta}^- \colon & (0,s_{\alpha,\beta}(\lambda^*)] \to (0,\lambda^*] , \quad\text{and}\quad && r_{\alpha,\beta}^+ \colon && (0,s_{\alpha,\beta}(\lambda^*)] \to [\lambda^*,1) , \\
& \lambda \mapsto (s_{\alpha,\beta}^-)^{-1}(\lambda) , && && \lambda \mapsto (s_{\alpha,\beta}^+)^{-1}(\lambda) .
\end{alignat*}
\end{itemize}
The cases $\alpha=0$ or $\beta=0$ can be treated analogously.

\end{document}